\documentclass[11pt,oneside]{amsart}

\title{Conjugacy for homogeneous ordered graphs}
\author{Samuel Coskey} \address{Samuel Coskey, Department of Mathematics, Boise State University, 1910 University Drive, Boise, ID, 83725}
\email{scoskey@gmail.com}
\urladdr{\href{http://scoskey.org}{scoskey.org}}
\author{Paul Ellis} \address{Paul Ellis, Department of Mathematics and Computer Science, Manhattanville College, 2900 Purchase Street, Purchase, NY, 10577}
\email{paulellis@paulellis.org}
\urladdr{\href{http://paulellis.org}{paulellis.org}}

\usepackage[marginratio=1:1]{geometry}
\usepackage{setspace}
\usepackage{amssymb,mathpazo,bm}
\usepackage{tikz}
\usepackage[textsize=footnotesize,color=blue!40,bordercolor=white]{todonotes}
\usepackage{hyperref}

\usepackage{etoolbox}
\makeatletter\pretocmd{\@seccntformat}{\S}{}{}
  \pretocmd{\@subseccntformat}{\S}{}{}\makeatother

\theoremstyle{definition}
\newtheorem{thm}{Theorem}[section]
\newtheorem{prop}[thm]{Proposition}
\newtheorem{lem}[thm]{Lemma}

\theoremstyle{remark}

\newtheorem{question}{Question}

\newcommand{\ZZ}{\ensuremath{\mathbb Z}}
\newcommand{\QQ}{\ensuremath{\mathbb Q}}

\newcommand{\tri}{\mathrel{\lhd}}
\DeclareMathOperator{\Aut}{Aut}

\begin{document}
\maketitle

\section{Introduction}

A countable relational structure $M$ is said to be \emph{homogeneous} if every finite partial automorphism of $M$ extends to an automorphism of $M$. The automorphism groups of homogeneous structures can have striking properties, see for instance \cite[Chapter~5]{macpherson}. In recent work \cite{summer,conjugacy1,conjugacy2}, we investigated the conjugacy classification problem for the automorphism groups of a variety of well-studied homogeneous structures such as graphs and digraphs. We found that with few exceptions, the conjugacy problem was of the maximum conceivable complexity.

In order to say what we mean by ``maximum conceivable complexity'', we very briefly recall the Borel complexity theory of equivalence relations. We refer the reader to \cite{gao} for more on this broadly applicable area of descriptive set theory. If $E,F$ are equivalence relations on standard Borel spaces $X,Y$, we say $E$ is \emph{Borel reducible} to $F$ if there exists a Borel function $f\colon X\to Y$ such that $x\mathrel{E}x'\iff f(x)\mathrel{F}f(x')$. If $Y$ is a space countable structures with isomorphism relation $\cong_Y$, we say that $\cong_Y$ is \emph{Borel complete} if for every space $X$ of countable structures with isomorphism relation $\cong_X$ we have that $\cong_X$ is Borel reducible to $\cong_Y$. For example the isomorphism relations on the classes of countable graphs, tournaments, and linear orders are all Borel complete.

Since conjugacy of automorphisms $f$ of a fixed structure $M$ is equivalent to isomorphism of expanded structures $(M,f)$, it makes sense to ask whether the conjugacy classification of automorphisms of $M$ is Borel complete. For most of the homogeneous structures $M$ that we considered in our recent work, we showed that the conjugacy classification of automorphisms of $M$ is Borel complete.

In this note we extend this family of results to the automorphisms of homogeneous ordered digraphs. A structure $M$ is said to be \emph{ordered} if one of the symbols of $M$ is a binary relation $<$ that satisfies the axioms of a linear order. Recently, Cherlin \cite{cherlin-ordered} classified the countable homogeneous ordered graphs. We will use this classification to establish our main result: for every ordered graph $G$ the conjugacy classification of automorphisms of $G$ is Borel complete.

Our proof strategy involves showing that each homogeneous ordered digraph possesses a very strong property called the ABAP. Before defining this property, we first recall that if $M$ is a homogeneous structure, then one commonly studies the class $\mathcal K$ of finite structures isomorphic to a substructure of $M$ and the class $\mathcal K_\omega$ of countable structures isomorphic to a substructure of $M$. These classes hold a variety of extension and amalgamation properties. (We refer the reader to \cite[Chapter~7]{hodges} for a starting point on amalgamation classes and related properties.) In our proofs, we will show that for each homogeneous ordered graph $G$, the corresponding class $\mathcal K_\omega$ satisfies a very strong kind of extension property. We then show that this property can be used to define a Borel reduction from a known Borel complete relation to the conjugacy relation on $\Aut(G)$.

We now define the extension property. Let $M$, $\mathcal K$, and $\mathcal K_\omega$ be as above. For any $A\in\mathcal K_\omega$ and quantifier-free type $\tau$ with finitely many parameters $\bar a$ from $A$, we say that $\tau$ is an \emph{admissible finite type} over $A$ if there exists $B\in\mathcal K_\omega$ such that $A\subset B$ and $B$ contains a witness for $\tau$. We say that $\mathcal K_\omega$ has the \emph{automorphic Borel amalgamation property} (ABAP) if there exists a Borel mapping $E$ on $\mathcal K_\omega$ such that the following hold:
\begin{enumerate}
  \item For any $A\in\mathcal K_\omega$, $E(A)$ is an extension of $A$ which contains witnesses for all admissable finite types over $A$;
  \item There is a Borel assignment $(A,\phi)\mapsto\tilde\phi$, from pairs $A\in\mathcal K_\omega$, $\phi\in\Aut(A)$ to $\tilde\phi\in\Aut(E(A))$ such that $\phi\subset\tilde\phi$ and $\tilde\phi$ has no fixed points in $E(A)\setminus A$;
  \item There is an assignment $(A_0,A_1,\alpha)\mapsto\hat\alpha$, from triples with $A_i\in\mathcal K_\omega$ and $\alpha\colon A_0\cong A_1$, to isomorphisms $\hat\alpha\colon E(A_0)\cong E(A_1)$ such that whenever $\phi_i\in\Aut(A_i)$ and $\alpha$ conjugates $\phi_0$ to $\phi_1$, we have that $\hat\alpha$ conjugates $\tilde\phi_0$ to $\tilde\phi_1$.
\end{enumerate}

These requirements may seem somewhat contrived, but the next result shows how each one is used. Moreover in our constructions below, each of (a)--(c) will be satisfied in a natural way.

\begin{thm}[\protect{\cite[Theorem~5.3]{conjugacy2}}]
  \label{theorem-consequence of ABAP}
  If $M$ is a homogeneous structure such that the class $\mathcal K_\omega$ of countable substructures of $M$ has the ABAP, then the isomorphism relation on $\mathcal K_\omega$ is Borel reducible to the conjugacy relation on $\Aut(M)$.
\end{thm}

The proof of this theorem is straightforward and we summarize it briefly. For $A\in K_{\omega}$, we iteratively apply (a) to obtain a structure $A_\infty$ which is a copy of $M$. We also iteratively apply (b) to the identity mapping of $A$ to obtain an automorphism $\phi_{A_\infty}$ of $A_\infty$. Then $A\mapsto\phi_{A_\infty}$ is the desired Borel reduction. The forward implication is shown using (c). For the reverse implication note that (b) implies the set of fixed points of $\phi_{A_\infty}$ is exactly $A$.

Our main result is the following.

\begin{thm}
  \label{theorem-main target}
  For every countable homogeneous ordered graph $G$, the collection $\mathcal K_\omega$ of countable substructures of $G$ has the ABAP.
\end{thm}

Since we will verify in Proposition~\ref{prop:substructures-bc} that all of the homogeneous ordered graphs have the property that the isomorphism relation on $\mathcal K_\omega$ is Borel complete, it follows from the two theorems together that for any homogeneous ordered graph $G$, the conjugacy relation on $\Aut(G)$ is Borel complete.

The rest of this article is organized as follows. In the next section we present Cherlin's classification of the countable homogeneous ordered graphs, and divide them into five types. We verify that the isomorphism classification of substructures of any homogeneous ordered graph is Borel complete, and also provide a general lemma about ordered graphs for later use. Then in sections 3--7 we establish the ABAP for each of the five types of homogeneous ordered graphs. For some of the types, we also establish the ABAP for the corresponding class of unordered structures as a ``warm-up''.

The work done in sections 3--7 may seem somewhat industrial, and one should hope that it can provide clues that will lead to a more general approach. While we have not yet found any substantially more efficient methods, the answer to any of the following questions would provide a step in that direction.

\begin{question}
  Are there versions of Lemma~\ref{lem:prec} for more general types of structures? Is the property described in the Lemma equivalent to some amalgamation or extension property?
\end{question}

\begin{question}
  How is the ABAP for a class of ordered structures related to the ABAP for the class of unordered counterparts? (That is, the reducts with the ordering removed.)
\end{question}

\begin{question}
  How is the ABAP related to any of several other approaches found in recent literature? See for instance \cite{kubis-masulovic}.
\end{question}

\textbf{Acknowledgement.} We would like to thank Greg Cherlin and Julien Melleray for helpful conversations about this material.

\section{The homogeneous ordered graphs and their properties}

We have said that Cherlin classified the countable homogeneous ordered graphs. Before reproducing the list, we first explain that several other types of structure can be viewed as ordered graphs. First, an ordered tournament $(T,<,\to)$ may be identified with an ordered undirected graph by setting $a\sim b$ whenever $a<b$ and $a\to b$. Thus an ordered linear order may be identified with an ordered graph too (these have been called model-theoretic \emph{permutations}, though we avoid the term). Second, a linear extension of a partial order $(P,<,\lhd)$ may be identified with an ordered undirected graph by setting $a\sim b$ whenever $a<b$ and $a\lhd b$.

With these identifications, we list the items in the classification as follows.
\begin{itemize}
  \item The generic ordered linear order
  \item The generic ordered local order (local orders are defined below)
  \item The generic ordered graph, the generic ordered $K_n$-free graphs, and complements of these
  \item The generic linear extension of the generic partial order, and its complement
  \item Several trivial examples including generic ordered empty and complete graphs, and generic ordered equivalence relations (described in detail below)
\end{itemize}

It is clear that the complementary forms will satisfy the ABAP if and only if the original forms do, so we will not address them any further.

\begin{prop}
  \label{prop:substructures-bc}
  If $G$ is a homogeneous ordered graph from the list above, then the classification of substructures of $G$ is Borel complete.
\end{prop}

As said in the introdution, this result together with Theorems~\ref{theorem-consequence of ABAP} and~\ref{theorem-main target} implies that the conjucacy relation on each $\Aut(G)$ is Borel complete.

\begin{proof}
  We recall the standard fact that the classification of linear orders is Borel complete. The mapping which sends $(L,<)$ to $(L,<,<)$ gives a Borel reduction from linear orders to ordered linear orders. Since every linear order is a local order, this covers the ordered local orders too. Since the structure $(L,<,<)$ is a linear extension of a partial order, this covers the ordered partial orders too.
  
  Next the mapping which sends $(L,<)$ to $(L,<,\perp)$ gives a Borel reduction from linear orders to ordered empty graphs. Since empty graphs are $K_n$-free, this coveres the generic ordered graph and $K_n$-free graphs too. Since the generic ordered equivalence relations all contain copies of the complete graph, they are covered as well.
\end{proof}

We close this section with two results about ordered and tournament structures that will be used several times in our constructions.

\begin{lem}
  \label{lem:prec}
  If $M$ is a linearly ordered structure, then there exists a linear ordering $\prec$ on the admissible finite types of $M$ with such that for any automorphisms $\phi$ of $M$ the natural extension of $\phi$ to the admissible finite types preserves the ordering $\prec$.
\end{lem}

\begin{proof}
  Each type is simply a quantifier-free sentence with one free variable and parameters $\bar a$. We can lexicographically preorder these types by their syntactic structure. Clearly this is always preserved by automorphisms. If $\tau(x,\bar a)$ and $\tau(x,\bar b)$ are syntactically equivalent, then we can order them using the lexicographic order on the parameters $\bar a$ and $\bar b$ induced by the order $<$ of the structure. Once again, since $<$ is a symbol of $M$, it is clear that this is preserved by automorphisms.
\end{proof}

We remark that we also have the analogous result for tournaments. That is, if $M$ is a tournament, then there exists a tournament structure on the admissible finite types of $M$ such that for any automorphisms $\phi$ of $M$ the natural extension of $\phi$ to the admissible finite types preserves the tournament structure. The proof is similar to that of the lemma.

\section{Ordered linear orders}

Before addressing the class of ordered linear orders, we briefly revisit the class of linear orders. The fact that the conjugacy relation on $(\QQ,<)$ is Borel complete was first established by Foreman \cite{foreman}. In our recent articles we have provided several slimmer proofs. We now give the ABAP version of the proof. Many of the arguments in later sections will rely on this proof as a template, with some details of the construction changed and the overall argument remaining the same.

\begin{thm}
  \label{thm:template1}
  The class of countable linear orders has the ABAP.
\end{thm}

\begin{proof}
  Let $(L,<)$ be a given countable linear order. We define an extension $L\subset E(L)$ which will contain witnesses for all admissible finite types over $L$. If $\tau$ is an admissible finite type, then $\tau$ is of the form $\{a_i<x,x<b_j\}$. We first note that it suffices to work only with types $\tau$ such that $\{a_i\}\neq\emptyset$, or else $\tau=\tau_0=\{x<b_0\}$ where $b_0$ is the minimum element of $L$ (if it exists). Indeed given any admissible type $\tau$, either one can extend $\tau$ so that $\{a_i\}\neq\emptyset$, or else any witness of $\tau_0$ also witnesses $\tau$. Thus a set of witnesses of types of this form will also be a set of witnesses for all admissible types.

  Now for each $\tau\neq\tau_0$ we may let $a_\tau=\max a_i$. We then add elements $x_{\tau,m}$ to $E(L)$ for $m\in\ZZ$ satisfying:
  \begin{itemize}
  \item $a_\tau<x_{\tau,m}$
  \item $x_{\tau,m}<b$ whenever $a_\tau<b\in L$
  \item $x_{\tau,m}<x_{\tau,m'}$ whenever $m<m'$
  \item if $a_\tau=a_{\tau'}$ then we set $x_{\tau,m}<x_{\tau',m'}$ if and only if $\tau\prec\tau'$, where $\prec$ is the order given by Lemma~\ref{lem:prec}
  \end{itemize}
  If $\tau=\tau_0$ we add elements $x_{\tau,m}$ to $E(L)$ satisfying:
  \begin{itemize}
  \item $x_{\tau,m}<b_0$
  \item $x_{\tau,m}<x_{\tau,m'}$ whenever $m<m'$
  \end{itemize}
  We additionally close the $<$ of $E(L)$ under transitivity. It is clear that condition~(a) of the ABAP is satisfied.

  Now let $\phi$ be an automorphism of $L$. Then $\phi$ naturally gives rise to a mapping of admissible finite types. We define the isomorphism $\tilde\phi$ of $E(L)$ to be the extension of $\phi$ such that $\tilde\phi(x_{\tau,m})=x_{\phi(\tau),m+1}$. It is clear from this definiton that condition~(b) of the ABAP is satisfied.

  Finally let $\alpha\colon L_0\to L_1$ be an isomorphism between linear orders. Then $\alpha$ naturally gives rise to a mapping from the admissible finite types of $L_0$ to those of $L_1$. We define the isomorphism $\hat\alpha\colon E(L_0)\to E(L_1)$ to be the extension of $\alpha$ such that $\hat\alpha(x_{\tau,m})=x_{\alpha(\tau),m}$. To verify condition~(c) of the ABAP, let $\phi_i\in\Aut(L_i)$ and suppose $\alpha\phi_0=\phi_1\alpha$. Then
  \[\hat{\alpha}\tilde{\phi}_0(x_{\tau,m})=x_{\alpha\phi_0(\tau),m+1}
  =x_{\phi_1\alpha(\tau),m+1}=\tilde{\phi}_1\hat{\alpha}(x_{\tau,m})
  \]
  Thus the ABAP is satisfied.
\end{proof}

We next address the class of ordered linear orders (again, these are also known as model-theoretic permuations).

\begin{thm}
  \label{thm:template2}
  The class of countable ordered linear orders has the ABAP.
\end{thm}

\begin{proof}
  Let $(L,<,\tri)$ be a given ordered linear order. We define an extension $L\subset E(L)$ which will contain all admissible finite types over $L$. If $\tau$ is an admissible finite type, then $\tau$ is of the form $\{a_i<x,x<b_j,c_k\tri x,x\tri d_l\}$. As in the previous proof either  $\{a_i\}\neq\emptyset$, or else $\{b_j\}=\{b_0\}$ where $b_0$ is the $<$-minimum element of $L$. Similarly either $\{c_k\}\neq\emptyset$, or else $\{d_l\}=\{d_0\}$ where $d_0$ is the $\lhd$-minimum element of $L$. For simplicity let us assume both $\{a_i\}\neq\emptyset$ and $\{c_k\}\neq\emptyset$; the remaining cases are not difficult to supply. (We may specify additonal needed relations between witnesses of different cases in some fixed fashion.)
  
  For such a $\tau$, let $a_\tau=\max_{<}a_i$, and let $c_\tau=\max_{\tri}c_k$. We will add a new element $x_{\tau,m}$ to $E(L)$ for each $m\in\ZZ$, with the relations given below. In the following, $\prec$ is the ordering given by Lemma~\ref{lem:prec}.
  \begin{itemize}
    \item $a_\tau<x_{\tau,m}$
    \item $x_{\tau,m}<b$ for all $b\in L$ such that $a_{\tau}<b$
    \item $c_{\tau}\tri x_{\tau,m}$
    \item $c_{\tau,m}\tri d$ for all $d\in L$ such that $c_{\tau}\tri d$
    \item if $m<m'$ then $x_{\tau,m}<x_{\tau,m'}$ and $x_{\tau,m}\tri x_{\tau,m'}$
    \item if $a_\tau=a_{\tau'}$ then we set $x_{\tau,m}<x_{\tau',m'}$ if and only if $\tau\prec\tau'$
    \item if $c_\tau=c_{\tau'}$ then we set $x_{\tau,m}\tri x_{\tau',m'}$ if and only if $\tau\prec\tau'$.
  \end{itemize}
  
Finally close $<$ and $\lhd$ under transitivity. It is possible to construct the $\tilde{\phi}$ and $\hat{\alpha}$ in the same fashion as the previous proof to verify that the ABAP holds.
\end{proof}

\section{Ordered local orders}

Before addressing ordered local orders, we consider plain local orders. Recall that a local order is a tournament such that for every vertex $x$, the set of predecessors of $x$ is linearly ordered and the set of successors of $x$ is linearly ordered. We have already shown that the generic local order has Borel complete conjugacy problem. In the following result we strengthen this and establish the ABAP for the class of countable local orders. Afterward we show how the argument can be modified to work for ordered local orders too.

\begin{thm}
  The class of countable local orders has the ABAP.
\end{thm}

\begin{proof}
  Let $(O,\to)$ be a given local order. We define an extension $O\subset E(O)$ as follows. Let $\tau$ be an admissible finite type of the form $a_i\to x\to b_j$, where $i\leq n$ and $j\leq m$. Since $O$ is a local order we may assume that $a_i\to a_{i+1}$ and $b_j\to b_{j+1}$. We assume for convenience that $\{a_i\}$ and $\{b_j\}$ are nonempty; the remaining cases are left as an exercise. Thus we may set $a_\tau=\max_\to a_i$ and $b_\tau=\max_\to b_j$.

  For each admissible finite type, we will add elements $x_{\tau,m}$ for each $\tau$ and $m\in\ZZ$. There are two cases to consider:
  \begin{enumerate}
  \item If $a_\tau\to b_\tau$, then place $x_{\tau,m}$ as an immediate successor to $a_\tau$. That is, set $a_\tau\to x_{\tau,m}$, and set $x_{\tau,m}\to y$ for all $y$ such that $a_\tau\to y$, and set $z\to x_{\tau,m}$ for all $z$ such that $z\to a_\tau$.
  \item If $b_\tau\to a_\tau$, then place $x_{\tau,m}$ as an immediate successor to the (nonexistent) anitpode of $b_\tau$. That is, set $x_{\tau,m}\to b_\tau$, and set $x_{\tau,m}\to y$ for all $y$ such that $y\to b_\tau$, and set $z\to x_{\tau,m}$ for all $z$ such that $b_\tau\to z$.
  \end{enumerate}
  
  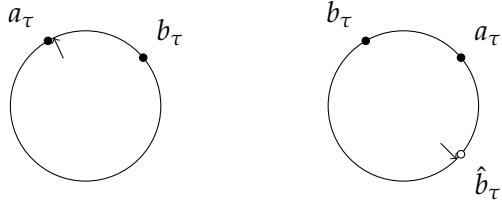
\begin{figure}[ht]
    \begin{tikzpicture}
      \draw (0,0) circle (1);
      \node[draw,circle,fill,inner sep=1pt,label=120:$a_\tau$] at (120:1) {};
      \node[inner sep=1pt,label=-40:\phantom{$\hat b_\tau$}] at (-40:1) {};
      \node[draw,circle,fill,inner sep=1pt,label=40:$b_\tau$] at (40:1) {};
      \draw[->] (115:.7) -- (115:1);
    \end{tikzpicture}
    \qquad\qquad
    \begin{tikzpicture}
      \draw (0,0) circle (1);
      \node[draw,circle,fill,inner sep=1pt,label=120:$b_\tau$] at (120:1) {};
      \node[draw,circle,fill=white,inner sep=1pt,label=-40:$\hat b_\tau$] at (-40:1) {};
      \node[draw,circle,fill,inner sep=1pt,label=40:$a_\tau$] at (40:1) {};
      \draw[->] (-45:.7) -- (-45:1);
    \end{tikzpicture}
    \caption{At left $a_\tau\to b_\tau$ and we insert the sequence of $x_{\tau,i}$ immediately clockwise from $a_\tau$. At right $b_\tau\to a_\tau$ and we insert the sequence $x_{\tau,i}$ immediately following the antipode of $b_\tau$.}
  \end{figure}
  
  In both cases, we set $x_{\tau,m} \to x_{\tau,m}$ for all $m<m'$. It remains to specify the edges between $x_{\tau,m}$ and $x_{\tau',m'}$ for $\tau\neq\tau'$. By the remark following Lemma~\ref{lem:prec}, we can find a tournament order $\prec$ on the admissible finite types. We set $x_{\tau,m}\to x_{\tau',m'}$ if and only if:
  \begin{itemize}
  \item $x_{\tau,m},x_{\tau',m'}$ are both in case (a), and either $a_\tau \to a_{\tau'}$, or $a_\tau = a_{\tau'}$ and $\tau \prec \tau'$
  \item $x_{\tau,m},x_{\tau',m'}$ are both in case (b), and either $b_\tau \to b_{\tau'}$, or $b_\tau = b_{\tau'}$ and $\tau \prec \tau'$
  \item $x_{\tau,m}$ is in case (a) and $x_{\tau',m'}$ is in case (b), and $b_{\tau'}\to a_{\tau}$
  \item $x_{\tau,m}$ is in case (b) and $x_{\tau',m'}$ is in case (a), and $a_{\tau'}\to b_{\tau}$.
  \end{itemize}
  Using the figure as a reference, it is not difficult to check that $E(O)$ is a local order.  For example, in Case (b) above, the predecessors of $x_{\tau,m}$ are preceisely the successors of $b_\tau$ together with all $x_{\tau,m'}$ so that $m'<m$, and this set is linearly ordered by $\to$.
 
  Finally, we define $\tilde\phi$ and $\hat\alpha$ as in the proof of Theorem~\ref{thm:template1}, and verify that the ABAP holds in a similar fashion.
\end{proof}

We can now build upon this method to establish the following.

\begin{thm}
  The class of ordered local orders has the ABAP.
\end{thm}

\begin{proof}
  Let $(O,<,\to)$ be an ordered local order.  Do the same thing as in the previous proof for the $\to$ relation. Do the same thing as for permutations for the $<$ relation.
\end{proof}

\section{Ordered generic graphs}

In this section we address ordinary generic graphs, as well as their ordered counterparts.

\begin{thm}
\begin{itemize}
\item The class of countable graphs has the ABAP. 
\item The class of countable $K_n$-free graphs has the ABAP.  
\item The class of countable ordered graphs has the ABAP. 
\item The class of countable ordered $K_n$-free graphs has the ABAP.
\end{itemize}
\end{thm}

\begin{proof}
We only prove the last statement.  The other three proofs are similar, perhaps with fewer details.

  Let $(G,<,\sim)$ be a given ordered graph with order relation $<$. We define an extension $G\subset E(G)$ which will contain all admissible finite types over $G$. Let $\tau(x)$ be an admissible finite type. Suppose the order relations in $\tau$ are enumerated as $\{a_i<x\}$ and $\{x<b_j\}$.  As before, we treat only the case where $\{a_i\}\neq \emptyset$, letting $a_\tau=\max_<\{a_i\}$. Next suppose the graph relations in $\tau$ are enumerated as $\{x\sim  c_k\}$ and $\{x\not\sim  d_l\}$.

  We let $E(G)$ consist of $G$ together with a new element $x_{\tau,m}$ for each $\tau$ and $m\in\ZZ$. We define the order and adjacency relations on $E(G)$ as follows.
  \begin{itemize}
  \item $a_\tau<x_{\tau,m}$
  \item $x_{\tau,m}<b$ for all $b\in G$ such that $a_{\tau}<b$
  \item if $m<m'$ then $x_{\tau,m}<x_{\tau,m'}$
  \item if $a_\tau=a_{\tau'}$ then we set $x_{\tau,m}<x_{\tau',m'}$ if and only if $\tau\prec\tau'$
  \item $x_{\tau,m}\sim c_k$ for all $k$, and $x_{\tau,m}$ is not adjacent to any other elements of $E(G)$
  \end{itemize}

  Note that each $\tau$ must not create a $K_n$.  Thus since we never place new edges between pairs of new vertices in $E(G)$, no $K_n$'s are created. Finally, we define $\tilde\phi$ and $\hat\alpha$ as in the proof of Theorem~\ref{thm:template1}.
%
%
%
%
\end{proof}

\section{Linear extensions of partial orders}

We have previously argued that the class of countable partial orders has the ABAP. However the following argument is not quite the same as the previous few. The reason is that this structure has a linear order which is not simply generic, but rather generic with respect to some property.

\begin{thm}
  The class of linear extensions of countable partial orders has the ABAP.
\end{thm}

\begin{proof}
  Let $(P,<,\lhd)$ be structure with $\lhd$ a partial order and $<$ a linear extension of it. We need to define an extension $E(P)$. Let $\tau(x)$ be an admissible finite type and suppose the linear order relations of $\tau$ are enumerated as $\{a_i<x\}$ and $\{x<b_j\}$. As before, we assume that $\{a_i\}\neq \emptyset$, and let $a_\tau=\max_<\{a_i\}$.

  Next suppose the partial order relations in $\tau$ are $\{c_k\lhd x\}$ and $\{x\lhd d_l\}$ and $\{x\perp_\lhd e_n\}$.  Without loss of generality $\{c_k\}\subseteq \{a_i\}$  and $\{d_l\}\subseteq \{b_j\}$.  To construct $E(P)$, for each such $\tau$, we add new points $\{x_{\tau, m}\}_{m\in\ZZ}$ satisfying:
  \begin{itemize}
  \item $x_{\tau, m} < x_{\tau, m+1}$
  \item $a_\tau < x_{\tau, m}$
  \item $x_{\tau, m}< b$ for each $b\in P$ such that $a_\tau<b$
  \item if $a_\tau=a_{\tau'}$ then we set $x_{\tau,m}<x_{\tau',m'}$ if and only if $\tau\prec\tau'$
  \item $c_k\lhd x_{\tau, m} \lhd d_l$
  \end{itemize}
  We then close $<$ and $\lhd$ under transitivity. Finally, we define $\tilde\phi$ and $\hat\alpha$ as in the proof of Theorem~\ref{thm:template1}.
\end{proof}

\section{And the rest, here on Gilligan's Isle!}

Let $\vec K_\infty$ be the structure obtained by ordering the complete graph $K_\infty$ with the order type of $\QQ$. The complementary structure $\vec I_\infty$ is the empty graph $I_\infty$ with order type $\QQ$. In each case, the class of substructures is isomorphic to the class of substructures of $(\QQ,<)$. Thus the proof that countable linear orders have the ABAP also establishes the following.

\begin{prop}
  The class of countable substructures of $\vec K_\infty$ has the ABAP. The class of countable substructures of $\vec I_\infty$ has the ABAP.
\end{prop}

Let $\vec I_\infty[\vec K_\infty]$ denote the ordered graph obtained by starting with $\vec I_\infty$ and substituting each point with a copy of $\vec K_\infty$. The substructures of $\vec I_\infty[\vec K_\infty]$ are the ordered equivalence relations with convex classes.

\begin{prop}
  The class of ordered equivalence relations with convex classes has the ABAP.
\end{prop}

\begin{proof}
  Let $(R,<,\sim)$ be a countable ordered equivalence relation with convex classes. We define $E(R)$ as follows. For our purposes, there are two kinds of admissible finite types, namely, whether the type specifies anything about $\sim$.  First suppose it does, and let $\tau=\{a_i<x, x<b_j, x\sim c \}$ be such a type. Leaving the other subcases as exercises, we again assume $\{a_i\}\neq \emptyset$, $\{b_j\}\neq \emptyset$,  and let $a_\tau=\max_< a_i$, $b_\tau=\min_< b_j$.  We add new points $x_{\tau,m}$ satisfying the following:   
  \begin{itemize}
  \item if $c\sim a_\tau$, set  $a_\tau<x_{\tau,m}$ and $x_{\tau,m}<b$ for all $a_\tau<b$
  \item if $c\not\sim a_\tau$, and $c\sim b_\tau$, set $x_{\tau,m}<b_\tau$ and $a<x_{\tau,m}$ for all $a<b_\tau$
  \item if $c\not\sim a_\tau$, and $c\not\sim b_\tau$, then we have $a_\tau<c<b_\tau$, so we can set $c<x_{\tau,m}$ and $x_{\tau,m}<b$ for all $c<b$
  \item $x_{\tau,m}<x_{\tau,m'}$ and $x_{\tau,m}\sim x_{\tau,m'}$ whenever $m<m'$
  \item for $y\in G$ set $x_{\tau,m} \sim y$ if and only if $y\sim c$
  \item set $x_{\tau,m} \sim x_{\tau',m'}$ if and only if $c=c'$
  \end{itemize}
  
  On the other hand, suppose $\tau=\{a_i<x, x<b_j \}$.  We again assume $\{a_i\}\neq \emptyset$  and let $a_\tau=\max_< a_i$.  We add new points $x_{\tau,m}$ satisfying the following:
  \begin{itemize}
  \item set  $a_\tau<x_{\tau,m}$ and $x_{\tau,m}<b$ for all $a_\tau<b$
  \item $x_{\tau,m}<x_{\tau,m'}$ and $x_{\tau,m}\sim x_{\tau,m'}$ whenever $m<m'$  
  \item for $y\in G$ set $x_{\tau,m} \not\sim y$.
  \item set $x_{\tau,m} \not\sim x_{\tau',m'}$ whenever $\tau\neq\tau'$
  \end{itemize}

  In each case, if the order between $x_{\tau,m}$ and $x_{\tau',m'}$ is not otherwise implied, set $x_{\tau,m}<x_{\tau',m'}$ if and only if $\tau\prec\tau'$. 
We then close $<$ and $\lhd$ under transitivity. We then close $\sim$ and $<$ under transitivity. Finally, we define $\tilde\phi$ and $\hat\alpha$ as in the proof of Theorem~\ref{thm:template1}.
\end{proof}


Note that the same proof applies to the complementary structure $\vec K_\infty[\vec I_\infty]$.

%

Finally, we have the shuffled product $\vec I_n*\vec K_\infty$ (these are type II.3 in Cherlin's classification). Here the graph is $n$ equivalence classes, and the ordering on each class has type $\QQ$, each dense in the whole thing.  The substructures of $\vec I_n*\vec K_\infty$ are ordered equivalence relations with at most $n$ equivalence classes.

\begin{prop}
  Let $n\leq \infty$. The class of ordered equivalence relations with $\leq n$ equivalence classes has the ABAP.
\end{prop}

\begin{proof}
  Fix $n\leq \infty$, and let $(L,<,\sim)$ be a countable ordered equivalence relation with $k (\leq n)$ equivalence classes. We define $E(L)$ as follows. Again there are two kinds of admissible finite types, namely, whether the type specifies anything about $\sim$.  First suppose it does, and let $\tau=\{a_i<x, x<b_j, x\sim c \}$ be such a type, and again assume $\{a_i\}\neq \emptyset$,  and let $a_\tau=\max_< a_i$.  We add new points $x_{\tau,m}$ satisfying the following: 
\begin{itemize}
  \item $a_\tau < x_{\tau, m}$ 
  \item $x_{\tau,m}< b$ for each $b\in L$ such that $a_\tau<b$
  \item $x_{\tau,m} < x_{\tau,m'}$ and $x_{\tau,m}\sim x_{\tau,m'}$ if $m<m'$
  \item if $a_\tau=a_{\tau'}$ then we set $x_{\tau,m}<x_{\tau,m'}$ if and only if $\tau\prec\tau'$
  \item $x_{\tau,m}\sim x_{\tau',m'}$ if and only if $c\sim c'$ 
  \item for $y\in L$, $x_{\tau,m}\sim y$ if and only if $y \sim c$
  \end{itemize}
 
On the other hand, suppose $\tau=\{a_i<x, x<b_j \}$, \emph{and} that $k<n$.  (All the points corresponding to types of this sort will together create one new equivalence class.) We again assume $\{a_i\}\neq \emptyset$  and let $a_\tau=\max_< a_i$.  We add new points $x_{\tau,m}$ satisfying the following:
  \begin{itemize}
  \item set  $a_\tau<x_{\tau,m}$ and $x_{\tau,m}<b$ for all $a_\tau<b$
  \item $x_{\tau,m}<x_{\tau,m'}$ whenever $m<m'$  
  \item for $y\in G$, set $x_{\tau,m} \not\sim y$.
  \item set $x_{\tau,m}\sim x_{\tau',m'}$ for all $\tau,\tau'$ of this category
  \end{itemize}
  We then close $\sim$ and $<$ under transitivity. As always, we define $\tilde\phi$ and $\hat\alpha$ as in the proof of Theorem~\ref{thm:template1}.
\end{proof}


\bibliographystyle{alpha}
\begin{singlespace}
  \bibliography{conjugacy}
\end{singlespace}

\end{document}